\documentclass[11pt]{amsproc}
\usepackage{amsthm, amsmath, amssymb}
\usepackage[margin=0.8in]{geometry}
\usepackage{hyperref,nicefrac}
\usepackage{cancel}

\numberwithin{equation}{section}
\numberwithin{figure}{section}

\newtheorem{thm}{Theorem}[section]

\newtheorem*{defn*}{Definition}

\newcommand{\setof}[1]{\left\{ {#1}\right\}}

\newcommand{\Aut}{\mbox{\rm Stab}}

\newcommand{\C}{{\mathbb{C}}}

\newcommand{\cC}{{\mathcal C}}

\newcommand{\cF}{{\mathcal F}}
\newcommand{\cG}{{\mathcal G}}

\newcommand{\cI}{{\mathcal I}}

\def\setof#1{\left\{{#1}\right\}}

\makeatother

\begin{document}

\title{The devil is in \xcancel{details} asymmetries}

\author{{Edinah K. Gnang} \and{Vidit Nanda}}
\begin{abstract}
We formally investigate some computational obstacles to tractability of computing the variety determined by $K$ complex polynomials in $N$ boolean variables. We show that using algebraic methods for solving combinatorial problems, the obstacles to tractability lies in the order of magnitude of asymmetries admitted by the given system of equations.
\end{abstract}

\maketitle

\section{Introduction}

Let $N$ and $K$ be natural numbers which remain fixed throughout the paper. Recall that $\C[x_1,\ldots,x_N]$ denotes the ring of complex polynomials in the $N$ variables $x_1, \ldots,x_N$, and let $\cI$ be the ideal generated by $\setof{x_n^2-x_n}_1^N$. Let $\cC = \nicefrac{\C[x_1,\ldots,x_n]}{\cI}$ be the ring of complex polynomials in $N$ Boolean variables. We investigate the tractability of computing the variety $Z(\cF)$ corresponding to the simultaneous zeros of $K$ given polynomials $\cF = \setof{f_k}_1^K \subset \cC$, each of which have at most $n$ terms. Recall that $\C[x_1,\ldots,x_N]$ -- and hence, $\cC$ -- admits a natural action of the symmetric group $\Sigma_N$ obtained by permuting the variables $\setof{x_n}_1^N$. More precisely, for each $\sigma \in \Sigma_N$ and $f \in \C[x_1,\ldots,x_N]$ we define $\sigma \circ f \in \cC$ by 
\[
\sigma\circ f(x_1, \ldots, x_n) = f(x_{\sigma(1)},\ldots,x_{\sigma(N)}).
\] 
Recall the polynomials fixed by this action of $\Sigma_N$ are called the {\em symmetric polynomials} of $\C[x_1,\ldots,x_N]$.

For each $\sigma \in \Sigma_N$, define the {\em $\sigma$-permuted system} $\cF_\sigma \subset \cC$ by $\cF_\sigma = \setof{\sigma \circ f_k}_1^N$. The {\em stabilizer} of the system $\cF$ is the subgroup of $\Sigma_N$ defined as follows
\[
\Aut(\cF) = \setof{\sigma \in \Sigma_N \mid \cF_\sigma = \cF}
\]

\section{The Polynomial method for solving Combinatorial problems with bounded size destabilizers}

\begin{thm}
If $\left|\Sigma_N \smallsetminus \Aut(\cF)\right|= c \le c_0 $ for some constant $c_0$ then the problem of determining $Z(\cF)$ is in co-NP. 
\end{thm}
\begin{proof}
We wish to determine whether the following algebraic variety is non-empty.
\begin{align*}
Z(\cF) = \setof{(x_1,\ldots,x_N) \in \setof{0,1}^N \mid f_k(x_1,\ldots,x_N) = 0 \text{ for each }1 \leq k \leq K}
\end{align*}
Consider the following iteration
\begin{equation}
\begin{array}{c}
f_{k,\left\{ \sigma_{1},\,\sigma_{2}\right\} }=\left(\sigma_{1}\circ f_k \right)\cdot \left(\sigma_{2}\circ f_k\right)\mod \cI\\
\vdots\\
f_{k,\left\{ \sigma_{0},\cdots,\,\sigma_{c-1} \right\} }=f_{k,\left\{ \sigma_{0},\cdots,\sigma_{c-2}\right\} }\cdot\left( \sigma_{c-1}\circ f_k \right)\mod \cI
\end{array}
\end{equation}
where 
\begin{equation*}
 \Sigma_N \smallsetminus \Aut(\cF) := \left\{ \sigma_{0},\cdots,\,\sigma_{c-1} \right\}.
\end{equation*}
So as to induce the following set of polynomials each having at most $n^c$ terms, 
all square free. 
\begin{equation}
\cG:=\left\{ f_{k,\left\{ \sigma_{0},\cdots,\,\sigma_{c-1} \right\} } \right\} _{0\le k<N}.
\end{equation}
We can determine a polynomial $p\left(x_{0}\right)$ in a single variable
\begin{equation}
p\left(x_{0}\right)=\prod_{0\le t<N}\left(x_{0}-\beta_{t}\right)
\end{equation}
in the ideal generated by the polynomials in $\cG$ i.e. $\exists\;\left\{h_k \right\} _{0\le k<n}\subset\cC$ such that
\begin{equation}
p\left(x_{0}\right)=\sum_{0\le k<K}h_k\cdot f_{k,\left\{ \sigma_{0},\cdots,\,\sigma_{c -1} \right\} }.
\end{equation}
Henceforth, let $\boldsymbol{\beta}$ denote the vector whose entries are the roots of the polynomial $p\left(x_{0}\right)$. If $\boldsymbol{\beta}$ admits an index $t$ for which $\beta_{t}\notin\mathbb{F}_{2}$
than this fact constitutes a certificate of non existence of solution
to $\cF$. However it would be incorrect to conclude that if $\forall\:0\le t<N$,
$\beta_{t}\in\mathbb{F}_{2}$ there should necessarily exist solutions
to $\cF$. The criteria invoked above is therfore a necessary but not a sufficient condition for
the existence of solution to $\cF$. The sufficient condition for the
existence of solution to $\cF$ is the fact that the matrix $\mathbf{M}_{\cF}$
of size $N\times c$, whose entries are given by 
\begin{equation}
\mathbf{M}_{\cF}:=\left(m_{k,\sigma}=\sigma\circ f_{k}\right)
\end{equation}
(where $0 \le k<K$ and $\sigma \in \Sigma_N \smallsetminus \Aut(\cF)$), has the property 
that $\mathbf{M}_{\cF}\mod\left(\mathbf{x}-\boldsymbol{\beta}\right)$, has at least one zero column.
\end{proof}

\end{document}